\documentclass[letter]{amsart}
\usepackage{amsmath, amsfonts, amssymb, amsthm}
\usepackage[bookmarks, bookmarksdepth=2, colorlinks=true, linkcolor=blue, citecolor=blue, urlcolor=blue]{hyperref}
\usepackage{color}
\usepackage{cleveref}
\usepackage{algorithm, algorithmicx}
\usepackage[noend]{algpseudocode}
\newcommand{\N}{\mathbb{N}}
\newcommand{\Z}{\mathbb{Z}}

\newcommand{\C}{\mathbb{C}}
\newcommand{\A}{\mathbb{A}}
\newcommand{\PP}{\mathbb{P}}
\DeclareMathOperator{\Jac}{Jac}

\DeclareMathOperator{\Gr}{Gr}
\newtheorem{prop}{Proposition}
\newtheorem{theorem}[prop]{Theorem}
\theoremstyle{definition}

\newtheorem{example}[prop]{Example}

\everymath{\displaystyle}

\begin{document}
\title{\bf{Numerical Implicitization}}

\author{Justin Chen}
\address{School of Mathematics, Georgia Institute of Technology,
Atlanta, Georgia, 30332 U.S.A.}
\email{jchen646@math.gatech.edu}

\author{Joe Kileel}
\address{Program in Applied and Computational Mathematics, Princeton University, Princeton, New Jersey, 08544 U.S.A.}
\email{jkileel@math.princeton.edu}

\subjclass[2010]{{14-04, 14Q99, 65H10, 65H20}}
\keywords{numerical algebraic geometry, implicitization, homotopy continuation, monodromy, interpolation}

\begin{abstract}
We present the {\textit{NumericalImplicitization}} package for {\textit{Macaulay2}}, which allows for user-friendly computation of the invariants of the image of a polynomial map, such as dimension, degree, and Hilbert function values.  This package relies on methods of numerical algebraic geometry, including homotopy continuation and monodromy.
\end{abstract}

\maketitle

\noindent
\textsc{Introduction.} Many varieties of interest in algebraic geometry and its applications are usefully described as images of polynomial maps, via a parametrization. Implicitization is the process of converting a parametric description of a variety into an intrinsic---or implicit---description. Classically, implicitization refers to the procedure of computing the defining equations of a parametrized variety, and in theory this is accomplished by finding the kernel of a ring homomorphism, via Gr\"obner bases. In practice however, symbolic Gr\"obner basis computations are often time consuming, even for medium scale problems, and do not scale well with respect to the size of the input.

Despite this, one would often like to know basic information about a parametrized variety, even when symbolic methods are prohibitively expensive (in terms of computation time). Examples of such information are discrete invariants such as the dimension, the degree, or Hilbert function values. Other examples include Boolean tests, for example whether or not a particular point lies on a parametrized variety. The goal of this \textit{Macaulay2} \cite{M2} package is to provide such information; in other words, to \textit{numerically implicitize} a parametrized variety by using methods of numerical algebraic geometry. 
\textit{NumericalImplicitization} builds on top of existing numerical algebraic geometry software: \textit{NAG4M2} \cite{Ley}, \textit{Bertini} \cite{Ber} and \textit{PHCpack} \cite{PHC}.  Each of these can be used for path tracking
and point sampling; by default the native software \texttt{M2engine} in \textit{NAG4M2} is used.  The latest version of the code and documentation can be found at
\url{https://github.com/Joe-Kileel/Numerical-Implicitization}.
\vspace{0.3cm}

\noindent
\textsc{Notation.} The following notation will be used throughout this article: 
\begin{itemize}
\item $X \subseteq \A^n$ is a \textit{source variety}, defined by an ideal $I = \langle g_{1}, \ldots, g_{r} \rangle$ in the polynomial ring $\C[x_1, \ldots, x_n]$
\item $F : \A^{n} \to \A^{m}$ is a regular map sending $x \mapsto (f_1(x), \ldots, f_m(x))$, where $f_i \in \C[x_1, \ldots, x_n]$
\item $Y$ is the Zariski closure of the image $\overline{F(X)} = \overline{F(V(I))} \subseteq \A^{m},$ the \textit{target variety} under consideration
\item $\widetilde{Y} \subseteq \PP^{m}$ is the projective closure of $Y$, with respect to the standard embedding $\A^{m} \subseteq \PP^{m}$.
\end{itemize}

Currently, \textit{NumericalImplicitization} is implemented for integral varieties $X$.  Equivalently, the ideal $I$ is prime. Since numerical methods are used, we always work with a floating-point representation for complex numbers. Moreover, $\widetilde{Y}$ is internally represented by its affine cone. This is because it is easier to work with affine, as opposed to projective, coordinates; at the same time, this suffices to find the invariants of $\widetilde{Y}$.
\vspace{0.3cm}

\noindent
\textsc{Sampling.} All the methods in this package rely on the ability to sample general points on $X$. To this end, the method {\tt{numericalSourceSample}} is provided to allow the user to sample general points on $X$. This method works by computing a witness set for $X$, via a numerical irreducible decomposition of $I$---once this is known, points on $X$ can be quickly sampled.

One way to view the difference in computation time between symbolic and numerical methods is that the upfront cost of computing a Gr\"obner basis is replaced with the upfront cost of computing a numerical irreducible decomposition, which is used to sample general points. However, if $X = \A^n$, then sampling is done by generating random tuples, so in this unrestricted (or rational) parametrization case, the upfront cost of numerical methods becomes negligible. Another situation where the cost of computing a numerical irreducible decomposition can be avoided is if the user can provide a single point on $X$: in this case, {\tt{numericalSourceSample}} can use the given point to quickly generate new general points on $X$ via path tracking.
\vspace{0.4cm}

\noindent
\textsc{Dimension.} The most basic invariant of an algebraic variety is its dimension. To compute the dimension of the image of a variety numerically, we use the following theorem \cite[III.10.4--10.5]{Har}:

\begin{theorem} Let $F : X \rightarrow Y$ be a dominant morphism of irreducible varieties over $\C$. Then there is a Zariski open subset $U \subseteq X$ such that for all $x \in U$, the induced map on tangent spaces $dF_x : T_xX \to T_{F(x)}Y$ is surjective. 
\end{theorem} \label{dimThm}



In the setting above, since the singular locus $\operatorname{Sing} Y$ is a proper closed subset of $Y$, for general $y = F(x) \in Y$ we
have that $\dim Y = \dim T_yY = \dim dF_x(T_xX) = \dim T_xX - \dim \ker dF_x$. Now $T_xX$ is the kernel of the Jacobian matrix of $I$ evaluated at $x$, given by $ \Jac(I)(x) = ((\partial g_i /\partial x_j)(x))_{1 \le i \le r, \, 1 \le j \le n}$, and $\ker dF_x$ is the kernel of the
 Jacobian of $F$ evaluated at $x$, intersected with $T_xX$.  Explicitly, $\ker dF_x$ is the kernel of the $(r+m) \times n$ matrix:

\[
\begin{bmatrix}
\Jac(I)(x) \\[5pt]
\Jac(F)(x)
\end{bmatrix} = 
\begin{scriptsize}
\begin{bmatrix}\\[-5.58pt]
\frac{\partial g_1}{\partial x_1}(x) & \ldots & \frac{\partial g_1}{\partial x_n}(x) \\[2pt]
\vdots & \ddots & \vdots \\[2pt]
\frac{\partial g_r}{\partial x_1}(x) & \ldots & \frac{\partial g_r}{\partial x_n}(x) \\[8pt]
\frac{\partial f_1}{\partial x_1}(x) & \ldots & \frac{\partial f_1}{\partial x_n}(x) \\[2pt]
\vdots & \ddots & \vdots \\[2pt]
\frac{\partial f_m}{\partial x_1}(x) & \ldots & \frac{\partial f_m}{\partial x_n}(x) \\[5pt]
\end{bmatrix}
\end{scriptsize}
\]
\vspace{0.05cm}

\noindent We compute these kernel dimensions numerically to obtain $\dim Y$.

\begin{example}
Let $Y \subseteq \textup{Sym}^{4}(\C^{5}) \cong \A^{70}$ be the variety of $5 \times 5 \times 5 \times 5$ symmetric tensors of border rank $\leq 14$. Equivalently, $Y$ is the affine cone over $\sigma_{14}(\nu_{4}(\PP^{4}))$, the $14^{\textup{th}}$ secant variety of the fourth Veronese embedding of $\PP^{4}$.  Naively, one expects $\textup{dim}(Y) = 14 \cdot 4 + 13 + 1 = 70$.  In fact, $\textup{dim}(Y) = 69$ as verified by the following code:

\vspace{0.05cm}

\begin{verbatim}
Macaulay2, version 1.13
i1 : needsPackage "NumericalImplicitization"
i2 : R = CC[s_(1,1)..s_(14,5)];
i3 : F = sum(1..14, i -> basis(4, R, Variables=>toList(s_(i,1)..s_(i,5))));
i4 : elapsedTime numericalImageDim(F, ideal 0_R)
     -- 0.0767826 seconds elapsed
o4 = 69
\end{verbatim}

\noindent This example is the largest exceptional case from the celebrated work \cite{AH}.
\end{example}

\vspace{0.2cm}

\noindent
\textsc{Hilbert function.} We now turn to the problem of determining the Hilbert function of $\widetilde{Y}$. If $\widetilde{Y} \subseteq \PP^{m}$ is a projective variety given by a homogeneous ideal $J \subseteq \C[y_{0}, \ldots, y_{m}]$, then the Hilbert function of $\widetilde{Y}$ at an argument $d \in \N$ is by definition the vector space dimension of the $d^\text{th}$ graded part of $J$, namely $H_{\widetilde{Y}}(d) := \dim J_d$. This counts the maximum number of linearly independent degree $d$ hypersurfaces in $\PP^{m}$ containing $\widetilde{Y}$. 

To compute the Hilbert function of $\widetilde{Y}$ numerically, we use \textit{multivariate polynomial interpolation}. For a fixed argument $d \in \N$, let $\{p_1, \ldots, p_N\}$ be a set of $N$ general points on $\widetilde{Y}$.  For $1 \le i \le N$, consider an $i \times \binom{m+d}{d}$ interpolation matrix $A^{(i)}$ with rows indexed by points $\{p_1, \ldots, p_i\}$ and columns indexed by degree $d$ monomials in $\C[y_{0}, \ldots, y_{m}]$, whose entries are the values of the monomials at the points. A vector in the kernel of $A^{(i)}$ corresponds to a choice of coefficients for a homogeneous degree $d$ polynomial that vanishes on $p_1, \ldots, p_i$. If $i$ is large, then one expects such a form to vanish on the entire variety $\widetilde{Y}$. The following theorem makes this precise:

\begin{theorem} \label{hilbertFunctionThm}
Let $\{p_1, \ldots, p_{s+1}\}$ be a set of general points on $\widetilde{Y}$, and let $A^{(i)}$ be the interpolation matrix above. If $\dim \ker A^{(s)} = \dim \ker A^{(s+1)}$, then $\dim \ker A^{(s)} = \dim J_d$.
\end{theorem}
\begin{proof}
Identifying $v \in \ker A^{(i)}$ with the form in $\C[y_{0}, \ldots, y_{m}]$ of degree $d$ having $v$ as its coefficients,
it suffices to show that $\ker A^{(s)} = J_d$. 
If $h \in J_d$, then $h$ vanishes on all of $\widetilde{Y}$, in particular on $\{p_1, \ldots, p_s\}$, so $h \in \ker A^{(s)}$. 
For the converse $\ker A^{(s)} \subseteq J_d$, we consider the universal interpolation matrices over $\C[y_{0,1}, \, y_{1,1}, \, \ldots, \, y_{m, i}]$:

$$\mathcal{A}^{(i)} := \begin{bmatrix} y_{0,1}^{d} & y_{0,1}^{d-1}y_{1,1} & \ldots & y_{m,1}^{d} \\[3pt] 
y_{0,2}^{d} & y_{0,2}^{d-1}y_{1,2} & \ldots & y_{m,2}^{d} \\[3pt]
\vdots & \vdots & \ddots & \vdots \\[3pt]
y_{0,i}^{d} & y_{0,i}^{d-1}y_{1,i} & \ldots & y_{m,i}^{d}
\end{bmatrix}$$

\vspace{0.15cm}
\noindent Set $r_{i} := \min \, \{r \in \Z_{\geq 0} \, | \, \textup{all } (r+1)\textup{-minors of } \mathcal{A}^{(i)} \textup{ lie in the ideal of } \widetilde{Y}^{\times i} \subseteq (\PP^{m})^{\times i} \}$.
Then any specialization of $\mathcal{A}^{(i)}$ to $i$ points in $\widetilde{Y}$ is a matrix over $\C$ of rank $\leq r_i$; moreover if the points are general, then the specialization has rank exactly $r_i$, since $\widetilde{Y}$ is irreducible. In particular $\textup{rank}(A^{s}) = r_s$ and $\textup{rank}(A^{s+1}) = r_{s+1}$, so $\dim \ker A^{(s)} = \dim \ker A^{(s+1)}$ implies that $r_s = r_{s+1}$.  
It follows that specializing $\mathcal{A}^{(s+1)}$ to $p_1, p_2, \ldots, p_s, q$ for \textit{any} $q \in \widetilde{Y}$ gives a rank $r_s$ matrix.  Hence, every degree $d$ form in $\ker A^{(s)}$ evaluates to 0 at every $q \in \widetilde{Y}$.  Since $\widetilde{Y}$ is reduced, we deduce that $\ker A^{(s)} \subseteq J_d$.
\end{proof}

It follows from \Cref{hilbertFunctionThm} that the integers $\dim \ker A^{(1)}, \dim \ker A^{(2)}, \ldots$ decrease by exactly $1$, until the first instance where they fail to decrease, at which point they stabilize: $\dim \ker A^{(i)} = \dim \ker A^{(s)}$ for $i \ge s$. This stable value is the value of the Hilbert function, $\dim \ker A^{(s)} = H_{\widetilde{Y}}(d)$. In particular, it suffices to compute $\dim \ker A^{(N)}$ for $N = \binom{m+d}{d}$, so one may assume the interpolation matrix is square. Although this may seem wasteful (as stabilization may have occurred with fewer rows), this is indeed what \texttt{numericalHilbertFunction} does, due to the algorithm used to compute kernel dimension numerically. To be precise, kernel dimension is found via a singular value decomposition (SVD)---namely, if a gap (the ratio of consecutive singular values) exceeds the option {\tt{SVDGap}} (with default value $10^5$), then this is taken as an indication that all singular values past this gap are numerically zero. On example problems, it was observed that taking only one more additional row than was needed often did not reveal a satisfactory gap in singular values.  In addition, numerical stability is improved via preconditioning on the interpolation matrices---namely, each row is normalized to have Euclidean norm $1$ before computing the SVD. Furthermore, for increased computational efficiency, the option {\tt UseSLP} allows for the usage of straight-line programs in creating interpolation matrices.

\begin{example} \label{canonicalCurveEx}
Let $X$ be a random canonical curve of genus 4 in $\PP^3$, so $X$ is the complete intersection of a random quadric and cubic.
Let $F : \PP^3 \dashrightarrow \PP^2$ be a projection by 3 random cubics.
Then $\widetilde{Y}$ is a plane curve of degree $3^{\dim(\widetilde{Y})} \cdot \deg(X) = 3 \cdot 2 \cdot 3 =18$,
so the ideal of $\widetilde{Y}$ contains a single form of degree 18.  We verify this as follows:

\begin{verbatim}
i5 : R = CC[w_0..w_3]; I = ideal(random(2,R), random(3,R)); F = toList(1..3)/(i -> random(3,R));
i8 : elapsedTime T = numericalHilbertFunction(F, I, 18, Verbose=>false)
     -- 6.01226 seconds elapsed
o8 : a numerical interpolation table, indicating
     the space of degree 18 forms in the ideal of the image has dimension 1
\end{verbatim}

\vspace{0.1cm}

The output is a \texttt{NumericalInterpolationTable}, which is a \texttt{HashTable} storing the results of the interpolation computation described above.  From this, one can obtain a floating-point approximation to a basis of $J_d$.  This is done via the command \texttt{extractImageEquations}:

\begin{verbatim}
i9 : extractImageEquations T
o9 : | -.0000712719y_0^18+(.000317507-.000100639i)y_0^17y_1- ... |
\end{verbatim}

\vspace{0.2cm}
\noindent
The option \texttt{AttemptExact=>true} calls the Lenstra-Lenstra-Lov\'{a}sz algorithm to compute short equations over $\Z$.
\end{example}

\vspace{0.2cm}

\noindent
\textsc{Degree.}
After dimension, degree is the most basic invariant of a projective variety $\widetilde{Y} \subseteq \PP^m$.
Set $k := \dim(\widetilde{Y})$. For a general linear space $ L \in \textup{Gr}(\PP^{m-k}, \PP^{m})$ of complementary dimension to $\widetilde{Y}$,
the intersection $L \cap \widetilde{Y}$ is a finite set of reduced points. The degree of $\widetilde{Y}$ is by definition the cardinality of $L \cap \widetilde{Y}$, which is independent of the general linear space $L$. Thus one way to find $\deg(\widetilde{Y})$ is to fix a random $L_0$ and compute the set of points $L_0 \cap \widetilde{Y}$.

\textit{NumericalImplicitization} takes this approach, but the method used to find $L_0 \cap \widetilde{Y}$ is not the most obvious. First and foremost, we do not know the equations of $\widetilde{Y}$, so all solving must be done in $X$.
Secondly, we do \textit{not} compute $F^{-1}(L_0) \cap X$ from the equations of $X$ and the equations of $L_0$ pulled back under $F$, 
because fibers of $F$ may be positive-dimensional and of high degree.
Instead, \textit{monodromy} is employed to find $L_0 \cap \widetilde{Y}$.

To state the technique, 
we consider the map:
$$ \{ (L, y) \in \Gr(\PP^{m-k}, \PP^{m}) \times \widetilde{Y} \,\, | \,\, y \in L\} \, \subseteq \Gr(\PP^{m-k}, \PP^{m}) \times \widetilde{Y} \xrightarrow{\makebox[1cm]{$\rho_1$}} \, \Gr(\PP^{m-k}, \PP^{m})$$
where $\rho_{1}$ is projection onto the first factor.  There is a nonempty Zariski open subset $U \subseteq \Gr(\PP^{m-k}, \PP^{m})$ such that the restriction $\rho_{1}^{-1}(U) \rightarrow U$ is a $\textup{deg}(\widetilde{Y})$-to-1
covering map, namely $U$ equals the complement of the Hurwitz divisor from \cite{St}.  For a fixed generic basepoint $L_{0} \in U$, the fundamental group $\pi_{1}(U, L_{0})$ acts on the fiber $\rho_{1}^{-1}(L_{0}) = L_{0} \cap \widetilde{Y}$. 
This action is known as monodromy.  It is a key fact that irreducibility of $\widetilde{Y}$ implies the group homomorphism $\pi_{1}(U, L_{0}) \longrightarrow \textup{Sym}(L_{0} \cap \widetilde{Y}) \cong \textup{Sym}_{\deg(\widetilde{Y})}$ is surjective (see \cite[Theorem A.12.2]{SW}).

We compute the degree of $\widetilde{Y}$ by constructing a \textit{pseudo-witness set} for $\widetilde{Y}$, which is a numerical representation of a parameterized variety (see \cite{HS}).
First, we sample a general point $x \in X$, and translate a general linear slice $L_0$ so that $F(x) \in L_0 \cap \widetilde{Y}$. Then $L_0$ is moved around in a random loop of the form described in \cite[Lemma 7.1.3]{SW}. This loop pulls back to a homotopy in $X$, where we use the equations of $X$ to track $x$. The endpoint of the track is a point $x' \in X$ such that $F(x') \in L_0 \cap \widetilde{Y}$. If $F(x)$ and $F(x')$ are numerically distinct, then the loop has \textit{learned} a new point in $L_0 \cap \widetilde{Y}$; otherwise $x'$ is discarded. We then repeat this process of tracking points in $X$ over each known point in $L_0 \cap \widetilde{Y}$, via new loops.
In practice, if several consecutive loops
do not learn new points in $L_0 \cap \widetilde{Y}$, then we suspect that all of $L_0 \cap \widetilde{Y}$ has been calculated. To verify this, we pass to the \textit{trace test} (see \cite[Corollary 2.2]{SVW}), which provides a characterization for when a subset of $L_0 \cap \widetilde{Y}$ equals $L_0 \cap \widetilde{Y}$. If the trace test is failed, then $L_0$ is replaced by a new random $L_0'$ and preimages in $X$ of known points of $L_0 \cap \widetilde{Y}$ are tracked to those preimages of points of $L_0' \cap \widetilde{Y}$. Afterwards, monodromy for $L'_0 \cap \widetilde{Y}$ begins anew. If the trace test is failed \texttt{MaxAttempts} (by default 5) times, then the method exits with only a lower bound on $\deg(\widetilde{Y})$. To speed up computation, the option \texttt{MaxThreads} allows for loop tracking to be parallelized.

\begin{example}\label{ex:raicu}
Let $\widetilde{Y} = \sigma_2(\PP^1 \times \PP^1 \times \PP^1 \times \PP^1 \times \PP^1) \subseteq \PP^{31}$. We find that $\deg(\widetilde{Y}) = 3256$, using the commands below:

\vspace{0.1cm}

\begin{verbatim}
i10 : R = CC[a_1..a_5, b_1..b_5, t_0, t_1];
i11 : F1 = terms product(apply(toList(1..5), i -> 1 + a_i));
i12 : F2 = terms product(apply(toList(1..5), i -> 1 + b_i));
i13 : F = apply(toList(0..<2^5), i -> t_0*F1#i + t_1*F2#i);
i14 : elapsedTime pseudoWitnessSet(F, ideal 0_R, Repeats=>2, MaxThreads=>2)
Sampling point in source ...
Tracking monodromy loops ...
Points found: 2
Points found: 4
...
Points found: 3256
Running trace test ...
      -- 336.737 seconds elapsed
o14 = a pseudo-witness set, indicating
      the degree of the image is 3256
\end{verbatim}

\vspace{0.1cm}

\noindent From \cite[Theorem 4.1]{Rai}, it is known that the prime ideal $J$ of $\widetilde{Y}$ is generated by the $3 \times 3$ minors of all flattenings of $2^{\times 5}$ tensors, so we can confirm that $\deg(J) = 3256$.  However, the naive attempt to compute the degree of $\widetilde{Y}$ symbolically by taking the kernel of a ring map---from a polynomial ring in 32 variables---has no hope of finishing in any reasonable amount of time.

\end{example}

\vspace{0.2cm}

\noindent \textsc{Membership.} 
Classically, given a variety $Y \subseteq \A^m$ and a point $y \in \A^m$, we determine whether or not $y \in Y$
by finding set-theoretic equations of $Y$ (which generate the ideal of $Y$ up to radical), 
and then testing if $y$ satisfies these equations.
If a \texttt{PseudoWitnessSet} for $Y$ is available, then point membership in $Y$ can instead be verified by \textit{parameter homotopy}.  More precisely, \texttt{isOnImage} determines if $y$ lies in the constructible set $F(X) \subseteq Y$, as follows.  We fix a general affine linear subspace $L_y \subseteq \A^m$ of complementary dimension $m- \dim Y$ passing through $y$.   Then $y \in F(X)$ if and only if $y \in L_y \cap F(X)$, so it suffices to compute the set $L_y \cap F(X)$. 
Now, a \texttt{PseudoWitnessSet} for $Y$ provides
a general section $L \cap F(X)$, and preimages in $X$.  We move $L$ to $L_y$ as in \cite[Theorem 7.1.6]{SW}.  This pulls back to a homotopy in $X$, where we use the equations of $X$ to track the preimages.  Applying $F$ to the endpoints of the track gives all isolated points in $L_y \cap F(X)$ by \cite[Theorem 7.1.6]{SW}.  Since $L_y$ was general, the proof of \cite[Corollary 10.5]{Eis} shows
$L_y \cap F(X)$ is zero-dimensional, so this procedure computes the entire set $L_y \cap F(X)$.

\begin{example}
Let $Y \subseteq \A^{18}$ be defined by the resultant of three quadratic equations in three unknowns.  In
other words, $Y$ consists of all coefficients $(c_1, \ldots, c_{6}, d_{1}, \ldots, d_{6}, e_{1}, \ldots, e_{6}) \in \A^{18}$ such that the system

\vspace{-0.15cm}

\[
\begin{aligned}
0 &= c_{1} x^{2} + c_{2} xy + c_{3} xz + c_{4} y^2 + c_{5} yz + c_{6} z^2 \\
0 &= d_{1} x^{2} + d_{2} xy + d_{3} xz + d_{4} y^2 + d_{5} yz + d_{6} z^2\\
0 &= e_{1} x^{2} + e_{2} xy + e_{3} xz + e_{4} y^2 + e_{5} yz + e_{6} z^2
\end{aligned}
\]

\vspace{0.05cm}

\noindent admits a solution $(x:y:z) \in \PP^2$. 
Here $Y$ is a hypersurface, and a matrix formula for its defining equation was derived in \cite{ES},
using exterior algebra methods.  We rapidly determine point
membership in $Y$ numerically as follows.

\begin{verbatim}
i15 : R = CC[c_1..c_6, d_1..d_6, e_1..e_6, x, y, z];
i16 : I = ideal(c_1*x^2+c_2*x*y+c_3*x*z+c_4*y^2+c_5*y*z+c_6*z^2,
                d_1*x^2+d_2*x*y+d_3*x*z+d_4*y^2+d_5*y*z+d_6*z^2,
                e_1*x^2+e_2*x*y+e_3*x*z+e_4*y^2+e_5*y*z+e_6*z^2);
i17 : F = toList(c_1..c_6 | d_1..d_6 | e_1..e_6);
i18 : W = pseudoWitnessSet(F, I, Verbose=>false); -- Y has degree 12
i19 : p1 = first numericalImageSample(F, I); p2 = point random(CC^1, CC^#F);
i21 : elapsedTime (isOnImage(W, p1), isOnImage(W, p2))
      -- used 0.186637 seconds
o21 = (true, false)

\end{verbatim}
\end{example}

\vspace{-0.15cm}

\noindent \textsc{Acknowledgements.} We are grateful to Anton Leykin for his encouragement, and to Luke Oeding for testing \textit{NumericalImplicitization}. We thank David Eisenbud and Bernd Sturmfels for helpful discussions, and the anonymous referee for insightful comments. This work has been partially supported by NSF DMS-1001867.

\end{document}